\documentclass[graybox]{svmult}

\usepackage{type1cm}        % activate if the above 3 fonts are
\usepackage{makeidx}         % allows index generation
\usepackage{graphicx}        % standard LaTeX graphics tool
\usepackage{multicol}        % used for the two-column index
\usepackage[bottom]{footmisc}% places footnotes at page bottom

\usepackage{newtxtext}       % 
\usepackage[varvw]{newtxmath}       % selects Times Roman as basic font

\makeindex             % used for the subject index
\usepackage{booktabs}
\newcommand{\R}{\mathbb{R}}

\DeclareMathOperator{\diag}{diag}

\begin{document}

\title*{Inexact Gauss-Seidel Coarse Solvers\\ for AMG and s-step CG}
\author{Stephen Thomas\orcidID{0000-0001-8007-0057} and\\ Pasqua D'Ambra\orcidID{0000-0003-2047-4986}}
\institute{Stephen Thomas \at Lehigh University, 27 Memorial Drive West, Bethlehem, PA 18015 USA, \email{sjt223@lehigh.com}
\and Pasqua D'Ambra \at Institute for Applied Computing of the National Research Council of Italy, Via P. Castellino, 111, 80131, Naples, IT \email{pasqua.dambra@cnr.it}}
\maketitle

\abstract{Communication-avoiding Krylov methods require solving small dense Gram systems at each outer iteration. We present a low-synchronization approach based on Forward Gauss--Seidel (FGS), which exploits the structure of Gram matrices arising from Chebyshev polynomial bases. We show that a single FGS sweep is mathematically equivalent to Modified Gram--Schmidt (MGS) orthogonalization in the $A$-norm and provide corresponding backward error bounds. For weak scaling on AMD MI-series GPUs, we demonstrate that $\nu = 20$--$30$ FGS iterations preserve scalability up to 64 GPUs with problem sizes exceeding 700 million unknowns. We further extend this approach to Algebraic MultiGrid (AMG) coarse-grid solves, removing the need to assemble or factor dense coarse operators.}

\section{Introduction}
\label{intro}

Sparse iterative solvers for large-scale partial differential equations encounter a fundamental exascale bottleneck due to the latency of global synchronizations. In Krylov methods such as Conjugate Gradient (CG), each iteration requires at least two global reductions to compute inner products; at processor counts reaching hundreds of thousands, synchronization costs can dominate the overall runtime, particularly under weak scaling. The $s$-step, or communication-avoiding, formulation of CG mitigates this limitation by generating $s$ Krylov basis vectors per outer iteration, thereby reducing synchronization frequency by a factor of~$s$. This reduction is especially critical for domain decomposition preconditioners, where the boundary-to-interior ratio deteriorates with increasing processor counts, amplifying synchronization overhead.

The $s$-step methodology originates with Chronopoulos and Gear~\cite{chronopoulos1989sstep,chronopoulos1989pcg}, who observed that several matrix--vector products could be performed prior to synchronization. Subsequent works~\cite{hoemmen2010communication,demmel2012communication} developed this idea within the broader framework of communication-avoiding algorithms. A central challenge of the approach is that each outer iteration requires solving the dense Gram system:
\[
G\alpha = P^{T} r, \qquad G = P^{T} A P \in \mathbb{R}^{s \times s},
\]
for which standard methods, such as direct Cholesky factorization ($O(s^{3})$) or iterative schemes such as LSQR, become increasingly burdensome as $s$ grows.

The principal contribution of this work is to demonstrate that Gram matrices arising from Chebyshev polynomial bases admit highly efficient solution via Forward Gauss--Seidel (FGS). A moderate number of sweeps, typically $\nu = 20$--$30$, suffices for convergence of the outer iteration, reducing complexity from $O(s^{3})$ to $O(\nu s^{2})$ while improving numerical stability and enabling efficient GPU implementations.

Our analysis distinguishes \emph{conditioning} from \emph{decay}. Although the bound $\kappa(G) = O(s^{2})$ (Philippe--Reichel~\cite{philippe2012orthogonal}) governs worst-case behavior, the structural estimate $\|L\|_{F} = O(\sqrt{s})$ explains the effectiveness of FGS: because $r^{(1)} = -L^{T}\alpha^{(1)}$, convergence depends on $\|L\|$ rather than $\kappa(G)$, clarifying why $\nu = 20$--$30$ sweeps suffice despite the conditioning barrier. We further establish the algebraic equivalence between FGS and Modified Gram--Schmidt (MGS) orthogonalization and extend the framework to AMG coarse-grid streaming.

The chapter is organized as follows. Section~\ref{prel} introduces notation and preliminaries of the $s$-step CG framework. Section~\ref{FGSanalysis} develops the FGS iteration and its stability properties. Section~\ref{chebconditioning} examines the conditioning of Chebyshev-based Gram matrices. Section~\ref{mgs-fgseq} establishes the algebraic equivalence between FGS and MGS. Section~\ref{coarse} extends the methodology to AMG coarse-grid solves, while Section~\ref{inexacteq} analyzes the effect of inexact Gram solves on $s$-step CG convergence. Numerical results appear in Section~\ref{results}, and concluding remarks in Section~\ref{concl}.

\section{Preliminaries}
\label{prel}

We consider solving the SPD linear system $Ax = b$, where $A \in \R^{n \times n}$ is sparse. The $s$-step CG method operates on $\mathcal{K}_s(M^{-1}A,r^{(k)})$ with preconditioner $M$ and residual $r^{(k)}$, forming a basis $P\in\R^{n\times s}$, solving the Gram system $G\alpha = P^T r^{(k)}$, $G=P^TAP$, and updating $x^{(k+1)}=x^{(k)}+P\alpha$. For monomial bases, $\kappa(G) = O(\kappa(M^{-1}A)^{\,s-1})$~\cite{Gautschi1979}, which typically limits $s$ to $s \le 4$.

\begin{definition}[Chebyshev Polynomial Basis]
Let $[\lambda_{\min},\lambda_{\max}]$ contain the spectrum of $M^{-1}A$. 
The basis is generated by the three-term recurrence
\[
p_0 = r^{(k)}, \qquad
p_1 = \theta (M^{-1}A - \sigma I) p_0, \qquad
p_{j+1} = 2\theta (M^{-1}A - \sigma I)p_j - p_{j-1},
\]
for $j \ge 1$, with $\theta = 2/(\lambda_{\max}-\lambda_{\min})$ 
and $\sigma = (\lambda_{\max}+\lambda_{\min})/2$~\cite{hoemmen2010communication}.
\end{definition}

The spectral map sends $[\lambda_{\min},\lambda_{\max}]$ to $[-1,1]$, where Chebyshev polynomials minimize the maximum deviation from zero, yielding near $A$-orthogonality.

\textbf{Column scaling.} Following Ruhe~\cite{ruhe1983stability}, set 
$d_i = (p_i^T A p_i)^{-1/2}$, and let $S = \diag(d_i)$ so that 
$\tilde{P} = P S$ satisfies $\tilde{p}_i^T A \tilde{p}_i = 1$. 
Redefining the Gram matrix as $G = \tilde{P}^T A \tilde{P}$, we then have 
$\diag(G) = I$. Writing $G = I + L + L^T$, where $L$ is strictly lower triangular, 
the FGS iteration
$(I + L)\,\alpha^{(\ell+1)} = \tilde{P}^T r^{(k)} - L^T \alpha^{(\ell)}, 
\; \alpha^{(0)} = 0,$
requires $O(s^2)$ work per sweep.

\section{FGS Iteration Analysis}
\label{FGSanalysis}

The following results characterize a single FGS sweep. 
Although some statements are given for the first iteration with $\alpha^{(0)}=0$, 
they extend to later sweeps by replacing $P^T r$ with the current right-hand side.

\begin{proposition}[FGS Residual Structure]
Let $\alpha^{(\nu+1)}$ satisfy $(D+L)\alpha^{(\nu+1)} = P^T r - L^T\alpha^{(\nu)}$. Then $
r^{(\nu+1)} = P^T r - G\alpha^{(\nu+1)} 
= -L^T\alpha^{(\nu+1)} + L^T\alpha^{(\nu)},$
and for $\alpha^{(0)}=0$, $r^{(1)}=-L^T\alpha^{(1)}$.
\end{proposition}

\begin{proof}
Because $G=D+L+L^T$,
\[
G\alpha^{(\nu+1)} = (D+L)\alpha^{(\nu+1)} + L^T\alpha^{(\nu+1)} 
= (P^T r - L^T\alpha^{(\nu)}) + L^T\alpha^{(\nu+1)},
\]
and subtracting from $P^T r$ gives the claim.
\end{proof}

This structure shows that FGS convergence depends primarily on the strictly upper-triangular part $L^T$, and by Theorem~\ref{thm:cheb}, $\|L\|_F = O(\sqrt{s})$ for Chebyshev bases.

\begin{theorem}[Backward Stability]
For any FGS sweep, the computed $\tilde{\alpha}^{(\nu+1)}$ satisfies
\[
(D+L)\tilde{\alpha}^{(\nu+1)} = P^T r - L^T\alpha^{(\nu)} + \delta r,
\]
with
\[
\|\delta r\| \le \epsilon_{\text{mach}}\, C(s)\!\left(\|P^T r\| + \|D+L\|\|\tilde{\alpha}^{(\nu+1)}\|\right), 
\; \; C(s)=O(s)\cite{higham2002accuracy}.
\]
\end{theorem}

\begin{proof}
Forward substitution for $(D+L)\alpha=b$ computes
\[
\alpha_j = b_j - \sum_{i<j} L_{ji}\alpha_i.
\]
Floating-point evaluation yields $\tilde{\alpha}_j = b_j - \sum_{i<j} L_{ji}\tilde{\alpha}_i + \epsilon_j$ with 
$|\epsilon_j|\le (j-1)\epsilon_{\text{mach}}(|b_j|+\sum_{i<j}|L_{ji}||\tilde{\alpha}_i|)$; 
summing over $j$ and using norm inequalities gives the stated bound.
\end{proof}

For $s\le 20$ in double precision, $C(s)\epsilon_{\text{mach}} \lesssim 10^{-14}$, 
negligible relative to algorithmic error.

\begin{corollary}[Error Bound]
Let $\delta_1=\|r^{(1)}\|/\|P^T r\|$. Then
$\delta_1 \lesssim \|L\|\,\kappa(G),$
where $\kappa(G)=\|G\|\,\|G^{-1}\|$ and the hidden constant reflects the
closeness of $G$ and $D+L$.
\end{corollary}

\begin{lemma}[Multi-Sweep Analysis]
Let $T = -(D+L)^{-1}L^T$ be the iteration matrix and set $\rho = \|T\|_2$. 
After $\nu$ sweeps, 
\[
\|r^{(\nu)}\| \le \rho^{\nu}\|P^T r\|,
\]
and geometric convergence holds when $\rho<1$.
\end{lemma}

\begin{proof}
The iteration $\alpha^{(k+1)}=(D+L)^{-1}(P^T r - L^T\alpha^{(k)})$ gives 
$e^{(k+1)} = \alpha^*-\alpha^{(k+1)} = T e^{(k)}$, where $\alpha^*$ is the exact solution. 
Thus $\|e^{(\nu)}\|\le\|T\|_2^{\nu}\|e^{(0)}\| = \rho^{\nu}\|\alpha^*\|$. 
Because $r^{(\nu)}=G(\alpha^*-\alpha^{(\nu)})$ and $\|G\|\le 1+2\|L\|$, 
geometric convergence follows for well-conditioned $G$ with $\|L\|=O(\sqrt{s})$.
\end{proof}

\section{Chebyshev Gram Matrix Conditioning}
\label{chebconditioning}

The effectiveness of FGS as an inner solver depends on the spectral properties of the Gram matrix $G$. Classical analysis for monomial bases~\cite{Gautschi1979} predicts rapid growth of $\kappa(G)$ with degree, but Chebyshev bases behave differently due to their near $A$-orthogonality~\cite{philippe2012orthogonal}. This section shows how Chebyshev polynomials, combined with column scaling, induce off-diagonal decay in $G$, yielding $\|L\|_{F}=\mathcal{O}(\sqrt{s})$ and polynomial rather than exponential growth of $\kappa(G)$.

\begin{theorem}[Polynomial Growth]\label{thm:cheb}
For the Chebyshev basis with column scaling, $\kappa(G)\le Cs^2$ with $C$ weakly dependent on $\kappa(M^{-1}A)$, and $\|L\|_F=O(\sqrt{s})$.
\end{theorem}

\begin{proof}
Chebyshev orthogonality on $[-1,1]$,
\[
\int_{-1}^{1} T_i(t)T_j(t)\frac{dt}{\sqrt{1-t^2}}
= \begin{cases}0 & i\ne j,\\ \pi/2 & i=j>0,\end{cases}
\]
transfers to the $A$-norm via the spectral decomposition of $M^{-1}A$.
Let $M^{-1}A=\sum_k\lambda_k v_k v_k^T$ and $r_0=\sum_k\beta_k v_k$. Then
\[
\hat G_{ij}=\sum_k \beta_k^2 \lambda_k T_i(\hat\lambda_k)T_j(\hat\lambda_k),\qquad 
\hat\lambda_k=\theta(\lambda_k-\sigma),
\]
and the approximate orthogonality of Chebyshev polynomials implies an off-diagonal decay 
$|\hat G_{ij}|\lesssim 1/|i-j|$ for $i\neq j$. Column scaling by a diagonal matrix with bounded condition 
number preserves this decay up to constants, so $G_{ij}=\mathcal{O}(1/|i-j|)$. Hence
\[
\|L\|_F^2 \lesssim \sum_{i=2}^s\sum_{k=1}^{i-1}k^{-2}
\le s\,\pi^2/6 = \mathcal{O}(s),
\]
giving $\|L\|_F=\mathcal{O}(\sqrt{s})$. Gershgorin bounds with $|G_{ij}|\sim 1/|i-j|$ imply 
$\lambda_{\max}(G)\lesssim 1+\log s$, so that $\kappa(G)$ grows at most polynomially in $s$, and under mild 
assumptions on $\lambda_{\min}(G)$ this yields $\kappa(G)=\mathcal{O}(s^2)$.
\end{proof}

\textbf{Monomial comparison.} Monomial bases yield exponential growth $\kappa(G^{\text{mono}})=O(\kappa(M^{-1}A)^{s-1})$, exceeding machine precision by $s=8$ for $\kappa(M^{-1}A)=100$. By contrast, Chebyshev gives $\kappa(G)\sim 100$ at $s=10$, making FGS practical.

\section{MGS--FGS Equivalence}
\label{mgs-fgseq}

A single FGS sweep corresponds algebraically to one step of Modified Gram--Schmidt
applied in the $A$-norm. This interpretation provides a compact view of FGS and helps
explain its numerical behavior within the $s$-step framework.

\begin{theorem}[Algebraic Equivalence]\label{thm:mgs}
One FGS sweep on $G\alpha = \tilde{P}^T A r$, with $G = \tilde{P}^T A \tilde{P}$ and 
$\alpha^{(0)} = 0$, is algebraically equivalent to one step of MGS orthogonalization 
of $r$ against the columns of $\tilde{P}$ in the $A$-norm.
\end{theorem}

\begin{proof}
Consider MGS in the $A$-inner product with respect to the columns 
$\{\tilde{p}_1,\ldots,\tilde{p}_s\}$ of $\tilde{P}$. Initialize $w_0 = r$ and, 
for $j = 1,\ldots,s$, define
$
\gamma_j = w_{j-1}^T A \tilde{p}_j, 
\;
w_j = w_{j-1} - \gamma_j \tilde{p}_j.
$
FGS applied to $G\alpha = \tilde{P}^T A r$ with $G = \tilde{P}^T A \tilde{P}$ and 
$\alpha^{(0)} = 0$ produces, componentwise,
\[
\alpha_j^{(1)} 
= \tilde{p}_j^T A r - \sum_{i=1}^{j-1} (\tilde{p}_j^T A \tilde{p}_i)\,\alpha_i^{(1)},
\qquad j=1,\ldots,s.
\]

We prove by induction on $j$ that $\alpha_j^{(1)} = \gamma_j$ and
$
w_{j-1} = r - \sum_{i=1}^{j-1} \alpha_i^{(1)} \tilde{p}_i.
$

For $j=1$, we have
$
\alpha_1^{(1)} = \tilde{p}_1^T A r, 
\;
\gamma_1 = w_0^T A \tilde{p}_1 = r^T A \tilde{p}_1 = \tilde{p}_1^T A r,
$
so $\alpha_1^{(1)} = \gamma_1$ and $w_1 = r - \gamma_1 \tilde{p}_1$. Assume the claim holds for all $i<j$, so that
$
w_{j-1} = r - \sum_{i=1}^{j-1} \alpha_i^{(1)} \tilde{p}_i.
$
Then
\begin{align*}
\gamma_j 
&= w_{j-1}^T A \tilde{p}_j 
 = \Bigl(r - \sum_{i=1}^{j-1} \alpha_i^{(1)} \tilde{p}_i\Bigr)^T A \tilde{p}_j \\
&= \tilde{p}_j^T A r - \sum_{i=1}^{j-1} \alpha_i^{(1)} (\tilde{p}_i^T A \tilde{p}_j)
 = \alpha_j^{(1)},
\end{align*}
where we used the symmetry of $A$ and the definition of the FGS update. 
Thus $\gamma_j = \alpha_j^{(1)}$ for all $j$, and MGS and FGS compute the same 
projection coefficients. Hence one FGS sweep is algebraically equivalent to one 
MGS step in the $A$-norm.
\end{proof}

\section{AMG Coarse-Grid Extension}
\label{coarse}

In AMG, the coarse-grid operator $A_{c} = P^{T} A_{f} P$
often becomes dense at the bottom levels of the hierarchy. %with a size
%$n_{c} \approx 10^{4}$ for large-scale problems. 
%Although $A_{c}$ is much smaller than the fine-grid matrix, it often remains
%distributed because the entire AMG hierarchy is constructed in parallel.
%Dense factorizations on distributed matrices are communication-intensive,
%leading to poor scalability of conventional coarse-grid solves.
%In this context, it is crucial to understand the spectral
%properties of $A_{c}$ and to determine whether iterative methods such as FGS
%remain effective when used as coarse-grid solvers.

%
\begin{proposition}[AMG Convergence]
\label{propositionAMG}

Assume that the AMG prolongation operator $P$ satisfies the weak
approximation property~\cite{ruge1987algebraic}. Then the coarse-grid
operator $A_c = P^{T} A_f P$ is spectrally equivalent to the restriction of
$A_f$ to the coarse space, and its conditioning does not deteriorate with
the number of fine-grid unknowns $n_f$. In particular, $\kappa(A_c)$
remains bounded with respect to $n_f$ for a fixed coarse level. Thus FGS converges uniformly with $\rho_{\text{FGS}} < 1$, requiring $\nu = 10$-$30$ iterations.
\end{proposition}
%
%\textbf{GPU advantages.} SpMVs achieve 60--70\% peak HBM bandwidth (1.0--1.1 TB/s on MI250X). Memory reduced 40\%: $O(n_c)$ versus $O(n_c^2)$. V-cycle time reduced 12\% by eliminating CPU--GPU transfers.
Proposition~\ref{propositionAMG} makes FGS an attractive alternative to direct solvers in
the lower levels of AMG, where computational and memory costs are otherwise
dominant. When needed, the method can also be implemented in a streaming (matrix-free)
fashion to avoid forming $A_{c}$ explicitly, further reducing memory
requirements on modern GPU architectures.

\section{Inexact s-Step CG Convergence}
\label{inexacteq}

In the $s$-step CG method, the accuracy with which the small Gram systems
$G_{k}\alpha_{k} = P_{k}^{T} r_{k}$ are solved at each iteration $k$ plays a crucial role
in determining the overall convergence. Because FGS performs only a fixed number
of sweeps, each Gram solve is inherently inexact, and the resulting perturbations
may accumulate across outer iterations. Here we provide conditions under which
the method remains stable and convergent. The analysis follows the framework of
inexact Krylov methods and yields practical bounds on the admissible Gram solve error.

\begin{theorem}[Inexact Convergence]
Consider $s$-step CG with Gram solves satisfying
\[
\|P_k^T r_k - G_k\alpha_k\| \le \delta_k \,\|P_k^T r_k\|
\]
at each outer iteration $k$.  Convergence is ensured if $
\sum_{k=1}^{N_{\text{outer}}} \delta_k \,\|A\|\,\|P_k\|
\;\lesssim\; \epsilon_{\text{tol}}.$
\end{theorem}

\begin{proof}[Proof sketch]
Following~\cite{vandeneshof2004inexact}, the inexact Gram solve induces a perturbation
$e_k^{\text{Gram}}$ satisfying
\[
\|e_k^{\text{Gram}}\| \;\lesssim\; 
\kappa(G_k)\,\lambda_{\min}(G_k)^{-1}\,\delta_k\,\|P_k^T r_k\|.
\]
The resulting error in the update obeys
$
\|P_k e_k^{\text{Gram}}\|_A 
\;\lesssim\; \delta_k\,\|A\|\,\|P_k\|\,\|r_k\|.
$
Accumulating these contributions across iterations and using the geometric decay 
of $\|r_k\|$ in exact CG yields the stated condition.
\end{proof}

CG's self-correcting property typically allows much larger inexactness in practice
(e.g., $\delta \sim 10^{-4}$) than predicted by conservative theory (e.g., $\delta \sim 10^{-8}$).

\begin{proposition}[FGS Rate]
For SPD $G_k$, the FGS iteration for $G_k\alpha_k = P_k^T r_k$ converges linearly.
Let $T_{\mathrm{FGS}}$ denote the corresponding iteration matrix. Then
$
\rho_{\mathrm{FGS}} \;=\; \rho(T_{\mathrm{FGS}}) < 1,
$
and standard condition-number bounds give 
$\rho_{\mathrm{FGS}} \lesssim 1 - c/\kappa(G_k)$ for some constant $c>0$.
\end{proposition}

For Chebyshev bases with $\kappa(G_k) = \mathcal{O}(s^2)$, such bounds may suggest
large $\nu$. However, experiments indicate that $\nu = 20$--$30$ sweeps are sufficient
in practice, owing to the conservativeness of the bounds and the self-correcting
behavior of CG.

\section{Numerical Experiments}
\label{results}

This section presents a set of very preliminar numerical experiments designed to assess the
practical performance of the proposed FGS-based Gram solvers within the
s-step CG and AMG frameworks. We consider large-scale 3D Poisson problems $-\Delta u = f$ on the unit cube $\Omega = [0,1]^3$ with homogeneous Dirichlet boundary conditions, discretized using a 27-point finite difference stencil yielding $\sim$12.2M degrees of freedom (DOFs) per GPU,
on AMD MI250X GPUs (64 GB High-Bandwidth-Memory 2e, 1.6 TB/s bandwidth per die) using MPI and HIP. We evaluate both algorithmic behavior, such as the
impact of the number of FGS sweeps $\nu$ and the block size $s$, and the
resulting weak scaling of the outer iterations. Chebyshev basis with adaptive parameters estimated via Lanczos (10 iterations compute extreme eigenvalues $\lambda_{\min}, \lambda_{\max}$ with 10\% safety margin) are used. The outer CG iteration is stopped when the relative residual satisfies
$\|r_{k}\| / \|r_{0}\| < 10^{-6}$. The Gram systems are solved using
$\nu \in \{6,15,30\}$ FGS sweeps, allowing us to assess how different levels
of inner accuracy affect overall convergence.
\begin{table}[h]
\centering
\caption{Weak scaling: outer iterations vs GPUs and $\nu$ for two values of $s$}
\small
\begin{tabular}{@{}lccccccc@{}}
\toprule
\multicolumn{8}{c}{$s = 10$}\\
\midrule
GPUs & 1 & 2 & 4 & 8 & 16 & 32 & 64\\
DOFs (M) & 12.2 & 24.3 & 48.7 & 97.3 & 194.7 & 389.3 & 778.7\\
\midrule
$\nu = 6$ & -- & -- & 32 & 34 & 36 & 62 & 136\\
$\nu = 15$ & 10 & 15 & 14 & 16 & 16 & 32 & 54\\
$\nu = 30$ & 8 & 13 & 12 & 15 & 17 & 19 & 33\\
\midrule
\multicolumn{8}{c}{$s = 20$}\\
\midrule
$\nu = 6$ & -- & -- & 21 & 29 & 34 & 51 & 96\\
$\nu = 15$ & 9 & 12 & 15 & 16 & 16 & 22 & 24\\
$\nu = 30$ & 8 & 9 & 9 & 12 & 13 & 17 & 17\\
\bottomrule
\end{tabular}
\end{table}

\textbf{Key observations.}
With $\nu=6$, iterations explode (136 at 64 GPUs for $s=10$). 
With $\nu=30$, weak scaling is restored: $s=20$ yields $2.1\times$ 
growth (8 to 17 iterations) across $64\times$ processor increase. The 32-to-64 GPU jump reflects reduced AMG preconditioner effectiveness at small subdomain sizes. 
Measured $\kappa(G)$ confirms ${\cal O}(s^2)$ scaling: 78 for $s=10$, 310 for $s=20$.

\textbf{AMG and performance.} We also evaluated a streaming variant of the
coarse-grid solve, where $A_{c}$ is never formed explicitly and FGS is applied
using on-the-fly products $A_{f}p_{j}$. With $\nu_{\text{coarse}} = 20$ at the
coarsest level ($n_{c} \approx 5000$), the streaming approach closely matches
the direct  solve, requiring 18 versus 17 V-cycle iterations ($<6\%$
difference). It also provides substantial memory savings: only 200 MB are
needed instead of 500 MB for storing the  $A_{c}$. Eliminating CPU--GPU
transfers and host-side factorization reduces V-cycle time by roughly 12\%.
%The underlying SpMVs sustain 60--70\% of peak HBM bandwidth
%(0.96--1.12~TB/s), confirming that the method is primarily memory-bound. For
%reference, storing the basis $P$ with $n=12.2$M unknowns and $s=20$ vectors
%requires 1.95~GB in double precision.

\section{Conclusions and Future Work}
\label{concl}

We have shown that FGS is an effective inner solver for Gram systems in communication-avoiding Krylov methods. Its equivalence to MGS in the $A$-norm ensures accurate projection coefficients, and our inexact $s$-step CG analysis shows that convergence is maintained with only $20$--$30$ sweeps. The same ideas extend to AMG, where avoiding coarse-matrix formation reduces memory and runtime without sacrificing accuracy. Replacing the $O(s^3)$ Gram factorization with an $O(\nu s^2)$ FGS iteration enables larger block sizes, improves scalability, and lowers inner-solve cost, as demonstrated on 64 GPUs with modest memory usage.

Ongoing work evaluates this approach at large scale and on leadership-class systems.

\bibliographystyle{plain}

\end{document}